\documentclass[11pt]{article}
\usepackage{amsmath,amsxtra,amssymb,color,latexsym,epsfig,amscd,amsthm,subfigure,fancybox,epsfig}
\usepackage[mathscr]{eucal}
\usepackage{bbm}
\usepackage{graphicx}
\usepackage{enumerate}
\usepackage{enumitem}
\usepackage{epsfig} 
\usepackage{epstopdf}
\usepackage{cases}
\newtheorem{thm}{Theorem}[section]
\newtheorem {asp}{Assumption}[section]
\newtheorem{lm}{Lemma}[section]
\newtheorem{prop}{Proposition}[section]

\theoremstyle{definition}
\newtheorem{defn}[thm]{Definition}

\theoremstyle{remark}
\newtheorem{rem}{Remark}[section]
\newtheorem{exam}{Example}[section]
\numberwithin{equation}{section}

\DeclareMathOperator{\trace}{tr}

\newcommand{\eps}{\varepsilon}

\newcommand{\p}{\mathfrak{p}}
\newcommand{\m}{\mathfrak{m}}
\newcommand{\C}{\mathcal{C}}
\newcommand{\D}{\mathcal{D}}

\newcommand{\M}{\mathcal{M}}
\newcommand{\F}{\mathcal{F}}

\newcommand{\E}{\mathbb{E}}

\newcommand{\LL}{\mathcal{L}}

\newcommand{\N}{{\mathbb{Z}}_+}

\newcommand{\Z}{\mathbb{Z}}

\newcommand{\PP}{\mathbb{P}}

\newcommand{\R}{\mathbb{R}}

\numberwithin{equation}{section}
\newcommand{\1}{\boldsymbol{1}}

\newcommand{\wdt}{\widetilde}

\newcommand{\bed}{\begin{displaymath}}
\newcommand{\eed}{\end{displaymath}}
\newcommand{\bea}{\bed\begin{array}{rl}}
\newcommand{\eea}{\end{array}\eed}
\newcommand{\ad}{&\!\!\!\disp}

\newcommand{\barray}{\begin{array}{ll}}
\newcommand{\earray}{\end{array}}

\def\disp{\displaystyle}
\newcommand{\rr}{{\Bbb R}}

\def\bar{\overline}
\def\hat{\widehat}
\def\a.s{\text{\;a.s.\;}}

\def\bnu{\boldsymbol{\nu}}

\begin{document}
\title{Stability of Regime-Switching Diffusion Systems with
Discrete
States Belonging to a Countable Set\thanks{This
research was supported in part by the National Science Foundation under grant DMS-1207667.}}
\author{Dang Hai  Nguyen\thanks{Department of Mathematics, Wayne State University, Detroit, MI
48202,
dangnh.maths@gmail.com.} \and
George Yin\thanks{Department of Mathematics, Wayne State University, Detroit, MI
48202,
gyin@math.wayne.edu.}}
\maketitle

\begin{abstract} This work focuses on
stability of regime-switching diffusions consisting of continuous and discrete components, in which the discrete
component switches in a countably infinite set and its switching rates at
 current time
 depend on the
 continuous component.
In contrast to the  existing approach,
this work provides more practically viable approach with more
feasible conditions
for stability. A classical approach for asymptotic stability
 using Lyapunov function techniques shows the
Lyapunov function evaluated at the solution process goes to 0 as time $t\to \infty$.
A distinctive feature of this paper is to obtain
 estimates of path-wise rates of convergence, which pinpoints how fast the aforementioned convergence to 0 taking place. 
 Finally,
some examples are given to illustrate our findings.



\end{abstract}

\newpage


\section{Introduction}\label{sec:int}
In the new era,
because of the pressing needs in networked systems
(including physical, biological, ecological, and social dynamic systems), large-scale optimization, and wired and wireless communications,
many new sophisticated control  systems have come into being.  Hybrid systems in which discrete and continuous states coexist and interact, are such a representative.
In particular, taking random disturbances into consideration, the so-called regime-switching diffusion
systems have drawn resurgent and increasing attentions.
A regime-switching diffusion is a two-component process $(X(t),\alpha(t))$, a continuous component and a discrete component taking values in a set
consisting of isolated points.
When the discrete component takes a value $i$ (i.e., $\alpha(t)=i$),
the continuous component $X(t)$ evolves according to the diffusion process whose drift and diffusion coefficients depend on $i$.
Asymptotic properties of
such systems such as stability have been studied intensively, 
because of numerous applications. For example, many issues such as
permanence, extinction, and persistence
 etc. of species in population dynamics and ecology are all linked to the stability issues.

Because many systems are in operation for a long period of time,
an important problem of great interest is the
stability of such systems.
Many results on different types of stability have been given for
switching diffusions when the state space of $\alpha(t)$ is finite (e.g., \cite{KZY, MY,XZ,YZW,YZ}).
Assuming that $\alpha(t)$ takes values in a countable state space,
 stability of the processes is more difficult to analyze.
To the best of our knowledge,
very few papers have considered stability of switching diffusion with countable switching states.
In \cite{SX}, some conditions for stability of those systems have been given
by approximating the generator of continuous state dependent switching process by that  of a Markov chain with finite state space.

To find sufficient conditions for stability,
it is desirable to find some common threads that are shared by
many specific systems.
Our motivation is based on the following thoughts.
First,
although the dynamics of $X(t)$
depend on the residence of the state of $\alpha(t)$,
the structures of equations for different
 states of $\alpha(t)$ are
 not drastically different but rather
 similar in certain sense.
This observation suggests finding a Lyapunov function that has similar form in different states of $\alpha(t)$.
For instance, suppose there is a Lyapunov function $V(x)$ such that in each
discrete state $i$, we have
$\LL_i V(x)\leq c_i V(x)$,
where $\LL_i$ is the generator of the diffusion in regime $i$
(more conditions and explanations for this inequality and related issues will be given in the next sections).
In this case, there is a common Lyapunov function shared by all the discrete states (or the Lyapunov function is independent of the discrete states).
It is well known that the sign of $c_i$ determines stability of the diffusion in each state $i$.
For the regime-switching diffusion,
one can expect that the stability of the system  depends not only on  $\{c_i\}$ but also on
the generator $Q(x)$ of the switching part.
A natural question is: under what relation between $\{c_i\}$ and $Q(x)$, the regime-switching diffusion
is stable?
When the number of regimes is finite, this question has been answered relatively completely (see \cite{KZY, SX}).
However, it is not straightforward to answer this question for the case of the discrete states belonging to a countable state space. We aim to take the challenges here.
Moreover, this paper also considers a generalization
when
the condition $\LL_i V(x)\leq c_i V(x)$ is replaced with
a condition of the type
$\LL_i V(x)\leq c_i g(V(x))$.

To date, muck work has been devoted to
the asymptotic stability  of diffusions and switching diffusions. A commonly used technique is based on Lyapunov stability argument.
For example, treating asymptotic stability,
much effort has been devoted to obtaining sufficient conditions under which the Lyapunov function evaluated at the solutions of the processes go to 0 as $t\to \infty$.
However,
 the question on how fast the Lyapunov function goes to 0 is unknown to date to the best of our knowledge.
 The current paper  settles this issue; it may be one of the first to provide a convergence rate  of the underlying process.
As another novel contribution, we
estimate the convergence rate of the solution to the equilibrium point by use of
 properties of the function $g(\cdot)$.

Treating switching diffusions as Markov processes, one may obtain
sufficient conditions for stability by
using a Lyapunov function satisfying certain properties.
However, 
the conditions are  often 
not directly related to the given system coefficients
(such as the drifts and diffusion matrices).
To obtain conditions that are based on coefficients of the systems,
we look at the issue of linearized (about the point of equilibrium) of the systems.
The idea is originated from the
topological equivalence of the linearized systems and original nonlinear systems due to the
 well-known Hartman-Grobman theorem in differential equations. Here in addition to linearizing the systems about the equilibrium point, we also replace $Q(x)$ by $Q(0)$.

The rest of this paper is organized as follows.
In Section \ref{sec:for}, we formulate
the equation for a regime-switching diffusion
and pose appropriate conditions for the existence and uniqueness
of solutions. We then provide the definitions
of certain types of stability
as well as give general conditions for the stability
of switching.
In Section \ref{sec:sta}, novel
and practical conditions for stability and instability of regime-switching diffusions
are given.
Applications of these conditions to linearizable systems
are given in Section \ref{sec:lin} and
examples are provided in Section \ref{sec:exa} to illustrate our findings.
Section \ref{sec:rem} is devoted to several remarks.
Finally, we provide the proofs of a number of technical results in an appendix.

\section{Formulation and Auxiliary Results}\label{sec:for}
Let $(\Omega,\F,\{\F_t\}_{t\geq0},\PP)$ be a complete filtered probability space with the filtration $\{\F_t\}_{t\geq 0}$ satisfying the usual condition,  i.e., it is increasing and right continuous while $\F_0$ contains all $\PP$-null sets.
Let $W(t)$ be an $\F_t$-adapted and $\R^d$-valued Brownian motion.
Suppose $b(\cdot,\cdot): \R^n\times\N\to\R^n$ and $\sigma(\cdot,\cdot): \R^n\times\N\to\R^{n\times d}$.
Consider the two-component process $(X(t),\alpha(t))$, where
 $\alpha(t)$ is a pure jump process taking value in  $\N= {\mathbb N} \setminus \{0\}=\{1,2,\dots\}$, the set of positive integers, and $X(t)\in\R^n$ satisfies

\begin{equation}\label{eq:sde}
dX(t)=b(X(t), \alpha(t))dt+\sigma(X(t),\alpha(t))dW(t).
\end{equation}

We assume that the jump intensity of $\alpha(t)$ depends on the current state of $X(t)$,
that is, there are functions $q_{ij}(\cdot):\R^n\to\R$ for $i,j\in\N$ satisfying
\begin{equation}\label{eq:tran}\begin{array}{ll}
&\disp \PP\{\alpha(t+\Delta)=j|\alpha(t)=i, X(s),\alpha(s), s\leq t\}=q_{ij}(X(t))\Delta+o(\Delta) \text{ if } i\ne j \
\hbox{ and }\\
&\disp \PP\{\alpha(t+\Delta)=i|\alpha(t)=i, X(s),\alpha(s), s\leq t\}=1-q_{i}(X(t))\Delta+o(\Delta).\end{array}\end{equation}
Throughout this paper, $q_{ij}(x)\geq0$ for each $i\ne j$
and $\sum_{j\in \N} q_{ij}(x)=0$ for each $i$ and all $x\in\R^n$. Denote
 $q_{i}(x)=\sum_{j=1,j\ne i}^\infty q_{ij}(x)$
(so $q_{ii}(x)=-q_i(x)$).
and $Q(x)=(q_{ij}(x))_{\N\times\N}$.
The process $\alpha(t)$ can be defined rigorously as the solution to a stochastic differential equation with respect to a Poisson random measure.
For each function $x\in\R^n, i\in\N$,  let $\Delta_{ij}(x), j\ne i$  be the consecutive left-closed, right-open intervals of the real line, each having length $q_{ij}(x)$.
That is,
\bea \ad \Delta_{i1}(x)=[0,q_{i1}(x)),\\
\ad \Delta_{ij}(x)=\Big[\sum_{k=1,k\ne i}^{j-1}q_{ik}(x),\sum_{k=1,k\ne i}^{j}q_{ik}(x)\Big), j>1, j\ne i.\eea
Define $h:\R^n\times\N\times\R\mapsto\R$ by
$h(x, i, z)=\sum_{j=1, j\ne i}^\infty(j-i)\1_{\{z\in\Delta_{ij}(x)\}}.$
The process $\alpha(t)$ can be defined as the solution to
$$d\alpha(t)=\int_{\R}h(X_t,\alpha(t-), z)\p(dt, dz)$$	
where $a(t-)=\lim\limits_{s\to t^-}\alpha(s)$ and $\p(dt, dz)$ is a Poisson random measure with intensity $dt\times\m(dz)$ and $\m$ is the Lebesgue measure on $\R$ such that
$\p(dt, dz)$ is
independent of the Brownian motion $W(\cdot)$.
The pair $(X(t),\alpha(t))$ is therefore a solution to
\begin{equation}\label{e2.3}
\begin{cases}
dX(t)=b(X(t), \alpha(t))dt+\sigma(X(t),\alpha(t))dW(t) \\
d\alpha(t)=\disp\int_{\R}h(X(t),
\alpha(t-), z)\p(dt, dz).\end{cases}
\end{equation}
A strong solution to  \eqref{e2.3} on $[0,T]$ with initial data $(x,i) \in \R^n \times \N$
is an $\F_t$-adapted process $(X(t),\alpha(t))$ such that
\begin{itemize}
  \item $X(t)$ is continuous and $\alpha(t)$ is cadlag (right continuous with left limits) with probability 1 (w.p.1).
  \item $X(0)=x$ and  $\alpha(0)=i_0$
  \item $(X(t),\alpha(t))$ satisfies \eqref{e2.3} for all $t\in[0,T]$ w.p.1.
\end{itemize}
Let  $f(\cdot,\cdot): \R^n\times\N\mapsto\R$ be  twice continuously differentiable in $x$.
We define the operator $\LL f(\cdot, \cdot): \R^n\times\N\mapsto \R$ by
\begin{equation}\label{eq:oper-def}
\begin{aligned}
\LL f(x, i)=&[\nabla f(x,i)]^\top b(x,i)+\dfrac12\trace\Big(\nabla^2 f(x,i)A(x,i)\Big) +\sum_{j=1,j\ne i}^\infty q_{ij} (x)
\big[f(x,j)-f(x,i)\big]\\
=&
\sum_{k=1}^nb_k(x,i)f_k(x,i)
+\dfrac12
\sum_{k,l=1}^na_{kl}(x,i)f_{kl}(x,i) +\sum_{j=1,j\ne i}^\infty q_{ij}(x)\big[f(x,j)-f(x,i)\big],
\end{aligned}
\end{equation}
where $\nabla f(x,i)=(f_1(x,i),\dots,f_n(x,i))\in \rr^{1\times n}$
and $\nabla^2 f(x,i)=(f_{ij}(x,i))_{n\times n}$ are the gradient and Hessian of $f(x,i)$ with respect to $x$, respectively,
 with \bea \ad f_k(x,i) =(\partial /\partial x_k) f(x,i),\
f_{kl} (x, i) = (\partial^2/\partial x_k \partial x_l) f(x,i),\
\hbox{ and }\\
\ad A(x,i)=(a_{ij}(x,i))_{n\times n}=\sigma(x,i)\sigma^\top(x,i),\eea
where $z^\top$ denotes the transpose of $z$.
If $(X(t),\alpha(t))$ satisfies \eqref{e2.3}, then by modifying the proof of \cite[Lemma 3, p.104]{AS},
we have the generalized It\^o formula:
\bed\label{e2.4}
\begin{aligned}
f(X(t),\alpha(t))-f(X(0),\alpha(0))=\int_0^t\LL f(X(s),\alpha(s-))ds+M_1(t)+M_2(t)
\end{aligned}
\eed
where $M_1(\cdot)$ and $M_2(\cdot)$ are two local martingales defined by
\begin{equation}\label{eq:M1-M2}
\barray \ad
M_1(t)=\int_0^t\nabla f(X(s),\alpha(s-))\sigma(X(s),\alpha(s-))dW(s),
\\
\ad
M_2(t)=\int_0^t\int_\R\big[f\big(X(s),\alpha(s-)+h(X(s),\alpha(s-), z)\big)-f(X(s),\alpha(s-))\big]\mu(ds,dz),
\earray \end{equation}
and $\mu(ds,dz)$ is the compensated Poisson random measure given by
$$\mu(ds,dz)=\p(ds,dz)- m(dz)ds .$$
Throughout this paper, we assume that either one of the following assumptions are satisfied.
Under either of them, it is proved in \cite{DY1} that \eqref{e2.3} has a unique solution with  given initial data.
Moreover, the solution is a Markov-Feller process.

\begin{asp}\label{asp4.1}{\rm
\begin{enumerate}\
  \item For each $i\in\N$, $H>0$, there is a positive constant $ L_{i,H}$ such that
$$|b(x,i)-b(y,i)|+|\sigma(y,i)-\sigma(x,i)|\leq  L_{i,H}|x-y|$$
if $x,y\in\R^n$ and $|x|$, $|y|\leq H$.
\item  For each $i\in\N$, there is a positive constant $\widetilde  L_{i}$ such that
$$|b(x,i)|+|\sigma(x,i)|\leq \widetilde  L_i(|x|+1).$$
\item $q_{ij}(x)$ is continuous in $x\in\R^n$ for each $(i,j)\in\N^2$. Moreover,
      $$M:=\sup_{x\in\R^n, i\in\N}\{|q_{i}(x)|\}<\infty.$$
\end{enumerate} }
\end{asp}

\begin{asp}\label{asp4.2}{\rm
\begin{enumerate}\
\item For each $i\in\N$, $H>0$, there is a positive constant $ L_{i,H}$ such that
$$|b(x,i)-b(y,i)|+|\sigma(x,i)-\sigma(y,i)|\leq  L_{i,H}|x-y|$$
if $x,y\in\R^n$ and $|x|,|y|\leq H$.
\item  There is a positive constant $\widetilde  L$ such that
$$|b(x,i)|+|\sigma(x,i)|\leq \widetilde  L(|x|+1).$$
\item $q_{ij}(x)$ is continuous in $x\in\R^n$ for each $(i,j)\in\N^2$. Moreover, for any $H>0$,
      $$M_H:=\sup_{x\in\R^n, |x|\leq H, i\in\N}\{|q_{i}(x)|\}<\infty.$$
\end{enumerate}}
\end{asp}

We suppose that $b(0,i)=0,\sigma(0,i)=0, i\in\N$
and give the following definitions of stability.
\begin{defn}
The trivial solution $X(t)\equiv0$
is said to be
\begin{itemize}
  \item {\it stable in probability}, if for any $h>0$,
  $$\lim_{x\to0} \inf_{i\in\N}
  \PP_{x,i}\big\{X(t)\leq h\,\forall\, t\geq0\big\}
  =1.$$
  \item {\it asymptotic stable in probability}, if it is stable in probability and
  $$\lim_{x\to0} \inf_{i\in\N}
  \PP_{x,i}\left\{\lim_{t\to\infty}X(t)=0\right\}
  =1.$$
\end{itemize}
\end{defn}

We state a general result that can be proved by well-known arguments;
see \cite[Section 7.2]{YZ}.

\begin{thm}\label{thm2.1}
Let $D$ be a neighborhood of $0\in\R^n$.
Suppose there exist three functions $V(x,i):D\times\Z\mapsto\R_+$, $\mu_1(x):D\mapsto\R_+$, $\mu_2(x):D\mapsto\R_+$
such that
\begin{itemize}
  \item $\mu_1(x),\mu_2(x)$ are continuous on $D$, $\mu_k(x)=0$ if and only if $x=0$ for $k=1,2$;
\item  $V(x,i)$ is continuous on $D$ and twice continuously differentiable in $\D\setminus\{0\}$ for each $i\in\N$;
\item $\mu_1(x)\leq V(x,i)$ for any $(x,i)\in D\times\N$.
\end{itemize}
Then the following conclusions hold.
\begin{itemize}
\item
if
$\LL V(x,i)\leq 0$ for any $(x,i)\in D\times\N$,
 the trivial solution is stable.
\item
if
$\LL V(x,i)\leq -\mu_2(x)$ for any $(x,i)\in D\times\N$
 the trivial solution is asymptotically stable.
\end{itemize}
\end{thm}

Let $\hat\alpha(t)$ be the Markov chain with bounded generator $Q(0)$
and transition probability $\hat p_{ij}(t)$

\begin{defn}
The Markov chain $\hat\alpha(t)$ is said to be
\begin{itemize}
  \item {\it ergodic}, if it has an invariant probability measure $\bnu=(\nu_1,\nu_2,\dots)$ and
  $$\lim_{t\to\infty}\hat p_{ij}(t)=\nu_j\text{ for any } i,j\in\N$$
or equivalently,
  $$\lim_{t\to\infty}\sum_{j\in\N}|\hat p_{ij}(t)-\nu_j|=0\text{ for any } i\in\N,$$
  \item {\it strongly ergodic}, if
  $$\lim_{t\to\infty}\sup_{i\in\N}\left\{\sum_{j\in\N}|\hat p_{ij}(t)-\nu_j|\right\}=0.$$
  \item {\it strongly exponentially ergodic}, if there exist $C>0$ and $\lambda>0$ such that
  \begin{equation}\label{see}
  \sum_{j\in\N}|\hat p_{ij}(t)-\pi_j|\leq Ce^{-\lambda t}\text{ for any } i\in\N, t\geq0.\end{equation}
\end{itemize}
\end{defn}
We refer to \cite{WA} for some properties and sufficient conditions for
the aforementioned ergodicity.

\section{Certain Practical Conditions for Stability and Instability}\label{sec:sta}
For each $h>0$, denote by $B_h\subset\R^n$ the open ball centered at $0$ with radius $h$.
Throughout this section, let $D$ be a neighborhood of $0$ satisfying $D\subset B_1$.
We also denote by $\hat\alpha(t)$ the continuous-time Markov chain with generator $Q(0)$.
Denote by $\LL_i$ the generator of the diffusion when the discrete component is in state $i$, that is,
$$\LL_iV(x)=\nabla V(x)b(x,i)+\dfrac12\trace\Big(\nabla^2 V(x)A(x,i)\Big).$$
We first state a theorem, which generalizes \cite[Theorem 4.3]{KZY}, a result
for switching diffusions when the switching takes values in a finite set.

\begin{thm}\label{thm3.1}
Suppose that the Markov chain $\hat\alpha(t)$
is strongly ergodic with invariant probability measure $\bnu=(\nu_1,\nu_2,\dots)$
and that
\begin{equation}\label{e1-thm3.1}
\sup_{i\in\N} \sum_{j\ne i}|q_{ij}(x)-q_{ij}(0)|\to 0 \,\text{ as }\, x\to0.
\end{equation}
Let  $D$ be a neighborhood of $0$ and $V:D\mapsto\R_+$
satisfying
that $V(x)=0$ if and only if $x=0$
and that $V(x)$ is continuous on $D$, twice continuously differentiable in $D\setminus\{0\}$.
Suppose that there is a bounded sequence of real numbers $\{c_i: i\in\N\}$
such that
\begin{equation}\label{e2-thm3.1}
\LL_i V(x)\leq c_iV(x) \,\forall\, x\in D\setminus\{0\}.
\end{equation}
Then, if $\sum_{i\in\N} c_i\nu_i<0$,
the trivial solution is asymptotic stable in probability.
\end{thm}

\begin{proof}
Let $\lambda=-\sum_{i\in\N} c_i\nu_i$.
Since $\sum_{i\in\N}\nu_i=1$,
we have $\sum_{i\in\N} (c_i+\lambda)\nu_i=0$.
Since $\hat\alpha(t)$ is strongly ergodic,
it follows from Lemma \ref{lm-a1} that there exists
a bounded  sequence of real numbers $\{\gamma_i: i\in\N\}$
such that
\begin{equation}\label{e3-thm3.1}
\sum_{j\in\N} q_{ij}(0)\gamma_j=\lambda+c_i \text{ for any } i\in\N
\end{equation}
Since $\sum_{j\in\N} q_{ij}(0)=0$ for any $i\in\N$ it follows from \eqref{e3-thm3.1}
that
\begin{equation}\label{e4-thm3.1}
\sum_{j\in\N} q_{ij}(0)\tilde\gamma_j=\sum_{j\in\N} q_{ij}(0)(1-p\gamma_j)=-p(\lambda+c_i) \text{ for any } i\in\N
\end{equation}
Since $\{\gamma_i\}$ is bounded,
we can choose $p\in(0,1)$ such that
\begin{equation}\label{e5-thm3.1}
p|\gamma_i|\leq \min\{0.25\lambda, 0.5\}
\end{equation}
In view of \eqref{e1-thm3.1} and \eqref{e5-thm3.1}, there is an
$h>0$  sufficiently small such that
\begin{equation}\label{e6-thm3.1}
\sum_{j\in\N}(1-p\gamma_j)|q_{ij}(x)-q_{ij}(0)|<\dfrac{p\lambda}4\quad\,\forall\,x\in B_h.
\end{equation}
Define the function $U(x,i):B_h\times\N\mapsto\R_+$ by $U(x,i)=(1-p\gamma_i)V^p(x)$.
By It\^o's formula, \eqref{e1-thm3.1}, \eqref{e4-thm3.1}, and \eqref{e6-thm3.1}, we have
\begin{equation}\label{e7-thm3.1}
\begin{aligned}
\LL U(x,i)=&p(1-p\gamma_i)V^{p-1}\LL_i V(x)-\dfrac{p(1-p)}2V^{p-2}\left|V_x(x)\sigma(x,i)\right|^2+V^p(x)\sum_{j\in\N}(1-p\gamma_j)q_{ij}(x)\\
\leq &c_ip(1-p\gamma_i)V^{p-1}+V^p(x)\sum_{j\in\N}(1-p\gamma_j)q_{ij}(0)+V^p(x)\sum_{j\in\N}(1-p\gamma_j)|q_{ij}(x)-q_{ij}(0)|\\
\leq &c_ip(1-p\gamma_i)V^{p-1}-p(\lambda+c_i)V^p(x)+V^p(x)\sum_{j\in\N}(1-p\gamma_j)|q_{ij}(x)-q_{ij}(0)|\\
\leq &p(-\lambda-p\gamma_i)V^p(x)+V^p(x)\sum_{j\in\N}(1-p\gamma_j)|q_{ij}(x)-q_{ij}(0)|\\
\leq &-0.75p\lambda V^p(x)+0.25p\lambda V^p(x)=-0.5p\lambda V^p(x)\,\text{ for } (x,i)\in B_h\times\N.
\end{aligned}
\end{equation}
By Theorem \ref{thm2.1}, it follows from \eqref{e7-thm3.1} that
the trivial solution is asymptotically stable.
\end{proof}
The hypothesis of this theorem seems to be restrictive.
It requires the strongly exponential ergodicity of $Q(0)$
and the uniform convergence to 0 of the sum $\sum_{j\ne i}|q_{ij}(x)-q_{ij}(0)|$.
To treat cases in which $Q(0)$ is strongly ergodic (not exponentially ergodic) or even only ergodic,
as well as to relax the condition \eqref{e1-thm3.1},
we need a more complicated method.
Our method, which is inspired by the idea in \cite{BL}, utilizes the ergodicity of $Q(0)$
and the analysis of the Laplace transform.
Similar techniques of using the Laplace transform can  also be seen
in the large deviations theory and related applications \cite{PB, ZWYJ}.
We also take a step further by estimating the pathwise rate of convergence
of solutions.

Let $\Gamma$ be a family of increasing and continuously differentiable functions $g:\R_+\mapsto\R_+$
such that $g(y)=0$ iff $y=0$.
Since $\frac{dg}{dy}(y)$
is
bounded on $[0,1]$ and $g(0)=0$,
it is easy to show that
the function
\begin{equation}\label{e-G}
G(y):=-\int_y^h\dfrac{dz}{g(z)}\quad\text{ on }\,[0,h]
\end{equation}
is non-positive and strictly decreasing and
$\lim_{y\to0}G(y)=-\infty$.
Its inverse $G^{-1}:(-\infty,0]\mapsto(0,h]$ satisfies
$$
\lim_{t\to\infty}G^{-1}(-t)=0.
$$
We state some assumptions to be used in what follows; we will also provide
some lemmas whose proofs are relegated to the appendix.

\begin{asp}\label{asp3.3}
There are functions $g\in \Gamma,$
$V:D\mapsto\R_+$ such that
\begin{itemize}
\item $V(x)=0$ if and only if $x=0$
\item $V(x)$ is continuous on $D$ and twice continuously differentiable in $D\setminus\{0\}$.
\item there is a bounded  sequence of real numbers $\{c_i: i\in\N\}$
such that
\begin{equation}\label{e2-thm3.3}
\LL_i V(x)\leq c_ig(V(x)) \,\forall\, x\in D\setminus\{0\}.
\end{equation}
\end{itemize}
\end{asp}

\begin{lm}\label{lm3.2}
Under Assumption \ref{asp3.3},
For any $\eps, T, h>0$, there exists an $\tilde h=\tilde h(\eps, T, h)$
such that
$$\PP_{x,i}\{\tau_h\geq T\}<\eps,\quad\text{for all }\, (x,i)\in B_{\tilde h}\times\N$$
where
$\tau_h=\inf\{t\geq 0: |X(t)|\geq h\}$.
\end{lm}

\begin{lm}\label{lm-a3}
Let $Y$ be a random variable, $\theta_0>0$ a constant, and suppose $$\E \exp(\theta_0 Y)+\E \exp(-\theta_0 Y)\leq K_1.$$
Then the log-Laplace transform
$\phi(\theta)=\ln\E\exp(\theta Y)$
is twice differentiable on $\left[0,\frac{\theta_0}2\right)$ and
$$\dfrac{d\phi}{d\theta}(0)= \E Y,\quad\text{ and }\,0\leq \dfrac{d^2\phi}{d\theta^2}(\theta)\leq K_2\,, \theta\in\left[0,\frac{\theta_0}2\right)$$
 for some $K_2>0$.
As a result of Taylor's expansion,
we have
$$
\phi(\theta)\leq\theta\E Y +\theta^2 K_2,\,\,\text{ for }\theta\in[0,0.5\theta_0).
$$
\end{lm}
\begin{lm}\label{lm3.3}
Under the assumption $b(0,i)=0,\sigma(0,i)=0, i\in\N$, we have
$$\PP_{x,i}\left\{X(t)=0\,\text{ for some }\, t\geq0\right\}=0
\,\text{ for any }\, x\ne 0, i\in\N.$$
\end{lm}
With the auxiliary results above, we can prove our main results.

\begin{thm}\label{thm3.2}
Suppose that the Markov chain $\hat\alpha(t)$
is ergodic with invariant probability measure $\bnu=(\nu_1,\nu_2,\dots)$
and Assumption  \ref{asp3.3}
is satisfied with
additional conditions:
\begin{equation}\label{e3-thm3.2}
\limsup_{i\to\infty}c_i<0,
\end{equation}
and
\begin{equation}\label{e4-thm3.2}
M_g:=\sup_{0<|x|<h, i\in\N}\left\{\left|\dfrac{V_x(x)\sigma(x,i)}{g(V(x))}\right|\right\}<\infty.
\end{equation}
Then, if $\sum_{i\in\N} c_i\nu_i<0$,
the trivial solution is asymptotic stable in probability, that is,
for any $h>0$ such that $B_h\subset D$,
and $\eps>0$, there exists $\delta=\delta(h,\eps)>0$ such that
$$\PP_{x,i}\left\{X(t)<h \,\forall\,t\geq0,\quad\text{and}\quad \lim_{t\to\infty}X(t)=0\right\}>1-\eps\,\text{ for any } (x,i)\in B_\delta\times\N.$$
Moreover,  there is a $\lambda>0$ such that
\begin{equation}\label{e-rate}
\PP_{x,i}\left\{\lim_{t\to\infty} \dfrac{V(X(t))}{G^{-1}(-\lambda t)}\leq 1\right\}>1-\eps\,\text{ for any } (x,i)\in B_\delta\times\N.
\end{equation}
\end{thm}

\begin{rem}{\rm
Before proceeding to the proof of the theorem, let us  make a brief comment.
In addition to providing sufficient conditions for asymptotic stability, a significant new element here is the rate of convergence given in \eqref{e-rate}.
Although there are numerous treatment of stochastic stability by a host of authors
for diffusions and switching diffusions. The rate result in Theorem \ref{thm3.2} appears to the first one of its kind.
}
\end{rem}

\begin{proof}$\text{}$
The proof is divided into
two steps.
We first show the trivial solution is stable in probability
and then we prove asymptotic stability
and  estimate the path-wise convergence rate.

{\bf Step 1: Stability.}

Let $h>0$ such that $B_h\subset D$. Since $\{c_i\}$ is bounded,
\begin{equation}\label{e5-thm3.2}
\lim_{k\to\infty}\sum_{i\leq k} c_i\nu_i=\sum_{i\in\N} c_i\nu_i<0.
\end{equation}
This and \eqref{e3-thm3.2}
show that there exists $k_0\in\N$ such that
$$-\lambda_1:=\sum_{i\leq k_0} c_i\nu_i<0$$
and
$$-2\lambda_2:=\sup_{i>k_0}c_i<0.$$
Let $\bar c=\sup_{i\in\N}|c_i|$ and $m_0$ be an positive integer satisfying
$m_0\lambda_2>\bar c+M_g+1$.
Define
$G(y)=-\int_y^hg^{-1}(z)dz.$
In view of Lemma \ref{lm3.3}, if $X(0)\ne 0$, then $X(t)\ne 0$ a.s,
which leads to $g(V(X(t))\ne 0$ a.s.
Thus, we have from It\^o's formula and the increasing property of $g(\cdot)$ that
\begin{equation}\label{e6-thm3.2}
\begin{aligned}
G\big(V(X(\tau_h\wedge t))\big)=&G(V(x))+
\int_0^{\tau_h\wedge t}\dfrac{\LL_{\alpha(s)} V(X(s))}{g(V(X(s)))}ds\\
&-\int_0^{\tau_h\wedge t}\dfrac{\dfrac{dg}{dy}(V(X(s)))\Big|V_x(X(s))\sigma(X(s),\alpha(s))\Big|^2}{2g^2(V(X(s)))}ds\\
&+\int_0^{\tau_h\wedge t}\dfrac{V_x(X(s))\sigma(X(s),\alpha(s))}{g(V(X(s)))}dW(s)
\leq G(V(x))+H(t),
\end{aligned}
\end{equation}
where
$$
\begin{aligned}
H(t)=&\int_0^{\tau_h\wedge t}c(\alpha(s))ds
+\int_0^{\tau_h\wedge t}\dfrac{V_x(X(s))\sigma(X(s),\alpha(s))}{g(V(X(s)))}dW(s).
\end{aligned}
$$
By It\^o's formula,
\begin{equation}\label{e-eH}
\begin{aligned}
e^{\theta H(t)}=& 1+\int_0^{t\wedge\tau_h}e^{\theta H(s)}\left[\theta c(\alpha(s))+\dfrac{\theta^2}2\dfrac{|V_x(X(s))\sigma(X(s),\alpha(s))|^2}{2g^2(V(X(s)))}\right]ds\\
\\&+
\theta\int_0^{t\wedge\tau_h}e^{\theta H(s)}\dfrac{V_x(X(s))\sigma(X(s),\alpha(s))}{g(V(X(s)))}dW(s),
\end{aligned}
\end{equation}
which leads to
$$
\begin{aligned}
\E_{x,i}e^{\theta H(t)}
=&1+\E_{x,i}\int_0^{\tau_h\wedge t}e^{\theta H(s)}\left[\theta c(\alpha(s))+\dfrac{\theta^2}2\dfrac{|V_x(X(s))\sigma(X(s),\alpha(s))|^2}{2g^2(V(X(s)))}\right]ds\\
\leq&1+ [\bar c+M_g]\E_{x,i}\int_0^{\tau_h\wedge t}e^{\theta H(s)}ds\\
\leq&1+[\bar c+M_g]\int_0^{t}\E_{x,i}e^{\theta H(s)}ds.
\end{aligned}
$$
In view of Gronwall's inequality, for any $t\geq0$ and $(x,i)\in B_h\times\N$, we have
\begin{equation}\label{e12-thm3.2}
\E_{x,i}e^{\theta H(t)}\leq e^{\theta[\bar c+M_g]t},\, \theta\in[-1,1].
\end{equation}
On the other hand,
we have
\begin{equation}\label{e7-thm3.2}
\begin{aligned}
\E_{x,i}H(t)&\leq\E_{x,i}\int_0^{\tau_h\wedge t}c(\alpha(s))ds\\
&\leq\E_{x,i}\int_0^{t}c(\alpha(s))ds-\E_{x,i}\int_{\tau_h\wedge t}^tc(\alpha(s))ds\\
&\leq\E_{x,i}\int_0^{t}c(\alpha(s))ds+t\bar c\PP_{x,i}\{\tau_h<t\}.
\end{aligned}
\end{equation}
Because of the ergodicity of $\hat\alpha(t)$,
there exists a $T>0$ depending on $k_0$ such that
\begin{equation}\label{e8-thm3.2}
\E_{0,i}\int_0^{t}c(\alpha(s))ds=\E_{i}\int_0^{t} c(\hat\alpha(s))ds\leq -\dfrac{3\lambda_1}4t\quad\,\forall\, t\geq T, i\leq k_0.
\end{equation}
By the Feller property of $(X(t),\alpha(t))$
there exists an $h_1\in(0,h)$ such that
\begin{equation}\label{e9-thm3.2}
\E_{x,i}\int_0^{t}c(\alpha(s))ds\leq -\dfrac{\lambda_1}2t\quad\,\forall\, t\in [T,T_2], |x|\leq h_1, i\leq k_0,
\end{equation}
where $T_2=(m_0+1)T$.
In view of Lemma \ref{lm3.2},
there exists an $h_2\in(0,h_1)$ such that
\begin{equation}\label{e10-thm3.2}
\bar c\PP_{x,i}\{\tau_h<m_0T\}\leq \dfrac{\lambda_1}4\,\text{ provided } |x|\leq h_2, i\in\N.
\end{equation}
Applying \eqref{e9-thm3.2} and \eqref{e10-thm3.2} to \eqref{e7-thm3.2},
we obtain
\begin{equation}\label{e11-thm3.2}
\E_{x,i}H(t)\leq -\dfrac{\lambda_1}4t \,\text{ if }\, 0<|x|\leq h_2, i\leq k_0, t\in[T,T_2].
\end{equation}
By Lemma \ref{lm-a3}, it follows from \eqref{e12-thm3.2} and \eqref{e11-thm3.2} that for $\theta\in[0,0.5], 0<|x|<h_2, i\leq k_0, t\in[T,T_2]$, we have
\begin{equation}\label{e13-thm3.2}
\begin{aligned}
\ln \E_{x,i}e^{\theta H(t)}\leq& \theta\E_{x,i}H(t)+\theta^2 K\\
\leq&-\theta\dfrac{\lambda_1 t}4+\theta^2 K
\end{aligned}
\end{equation}
for some $K>0$ depending on $T_2, \bar c$ and $M_g$.
Let $\theta\in(0,0.5]$ such that
\begin{equation}\label{e14-thm3.2}
\theta K<\dfrac{\lambda_1 T}8,\,\text{ and }\, \theta M_g<\lambda_2
\end{equation}
we have
$$
\ln \E_{x,i}e^{\theta H(t)}\leq -\dfrac{\theta\lambda_1 t}8\,\text{ for }0<|x|<h_2, i\leq k_0, t\in[T,T_2]
$$
or equivalently,
\begin{equation}\label{e15-thm3.2}
\E_{x,i}e^{\theta H(t)}\leq \exp\left\{-\dfrac{\theta\lambda_1 t}8\right\}\,\text{ for }0<|x|<h_2, i\leq k_0, t\in[T,T_2].
\end{equation}
In what follows, we fix a $\theta>0$ satisfying \eqref{e15-thm3.2}.
Exponentiating both sides of the inequality
$G(V(X(\tau_h\wedge t)))\leq G(V(x))+H(t)$
we have
 for $0<|x|<h_2, i\leq k_0, t\in[T,T_2]$ that
\begin{equation}\label{e16-thm3.2}
\E_{x,i}U(X(\tau_h\wedge t))
\leq U(x)\E_{x,i}e^{\theta H(t)}
\leq U(x)\exp\left\{-\dfrac{\theta\lambda_1 t}8\right\}.
\end{equation}
where $U(x)=\exp(\theta G(V(x))).$
Since $\lim_{x\to0}G(V(x))=-\infty$
then
\begin{equation}\label{e27-thm3.2}
\lim_{x\to0} U(x)=0.
\end{equation}
Using the inequality $G(V(X(\tau_h\wedge t)))\leq G(V(x))+H(t)$
and \eqref{e12-thm3.2} we have
\begin{equation}\label{e26-thm3.2}
\E_{x,i}U(X(\tau_h\wedge t))
\leq U(x)\exp\left\{\theta[\bar c+M_g]t\right\}, \text{ for all } (x,i)\in B_h\times\N, t\geq0.
\end{equation}
Now,
let $\Delta=\inf\{U(x): h_2\leq |x|\leq h\}>0$.
Define stopping times
$$\xi=\inf\{t\geq0: \alpha(t)\leq k_0\},\quad\text{ and }\, \zeta=\inf\{t\geq0: U(X(t))\geq\Delta\}.$$
Clearly, if $X(0)\in B_h$ then $\zeta\leq\tau_h$ and if $t<\zeta$ then $|X(t)|<h_2$.
By computation and \eqref{e14-thm3.2}, we have
$$
\begin{aligned}
\LL_i U(x)\leq\theta U(x)\left[c_i+[\theta-\dot g(V(x)]\dfrac{|V_x(x)\sigma(x,i)|^2}{g(V(x))}\right]
\leq& \theta(-2\lambda_2+\theta M_g)U(x)\\
\leq& -\theta\lambda_2U(x),\,\text{ for } 0<|x|<h, i>k_0.
\end{aligned}
$$
It follows from  It\^o's formula that
\begin{equation}\label{e17-thm3.2}
\begin{aligned}
\E_{x,i} e^{\theta\lambda_2(t\wedge\xi\wedge\zeta)}U(X(t\wedge\xi)
=&U(x)+\E_{x,i}\int_0^{t\wedge\xi\wedge\zeta}e^{\lambda_2s}\left[\theta\lambda_2U(X(s))+\LL_{\alpha(t)}U(X(s))\right]ds\\
\leq &U(x),\,\text{ for } 0<|x|<h, i\in\N.
\end{aligned}
\end{equation}
We have the following estimate for $0<|x|<h, i>k_0$.
\begin{equation}\label{e18-thm3.2}
\begin{aligned}
\E_{x,i} e^{\theta\lambda_2(T_2\wedge\xi\wedge\zeta)}U(X(T_2\wedge\xi\wedge\zeta))
=&\E_{x,i} \1_{\{\xi\wedge\zeta<m_0T\}}e^{\theta\lambda_2(T_2\wedge\xi\wedge\zeta)}U(X(T_2\wedge\xi\wedge\zeta))\\
&+\E_{x,i} \1_{\{m_0T\leq \xi\wedge\zeta<T_2\}}e^{\theta\lambda_2(T_2\wedge\xi\wedge\zeta)}U(X(T_2\wedge\xi\wedge\zeta))\\
&+\E_{x,i} \1_{\{\xi\wedge\zeta\geq T_2\}}e^{\theta\lambda_2(T_2\wedge\xi\wedge\zeta)}U(X(T_2\wedge\xi\wedge\zeta))\\
\geq&\E_{x,i} \1_{\{\xi\wedge\zeta\leq m_0T\}}U(X(\xi\wedge\zeta))\\
&+e^{\theta\lambda_2m_0T}\E_{x,i} \1_{\{m_0T\leq \xi\wedge\zeta<T_2\}}U(X(\xi\wedge\zeta))\\
&+e^{\theta\lambda_2T_2}\E_{x,i} \1_{\{\xi\geq T_2\}}U(X(T_2)).
\end{aligned}
\end{equation}
Since $\PP_{x,i}\{\zeta=0\}=1$ if $i\leq k_0$,
\eqref{e18-thm3.2} holds for $0<|x|<h, i\in\N$.
Noting that $U(x)\wedge\Delta\leq\Delta$ for any $x\in B_h$, we have
$$\E \Big[U(X(T_2\wedge\tau_h))\wedge\Delta\Big|\zeta<m_0T,\zeta\leq\xi\Big]\leq \Delta\leq U(X(\zeta))=U(X(\xi\wedge\zeta)).
$$
If $\xi<\zeta$, then $U(X(\xi))<\Delta$.
By strong Markov property of $(X(t),\alpha(t))$, \eqref{e16-thm3.2}, and \eqref{e12-thm3.2}, we have
$$
\E \Big[U(X(T_2\wedge\tau_h))\wedge\Delta\Big|\xi<m_0T\wedge\zeta\Big]\leq U(X(\xi))=U(X(\xi\wedge\zeta))$$
and
$$\E \Big[U(X(T_2\wedge\tau_h))\wedge\Delta\Big|m_0T\leq\xi<T_2\wedge\zeta\Big]\leq U(X(\xi))e^{\theta (\bar c+M_g)T}=U(X(\xi\wedge\zeta))e^{\theta (\bar c+M_g)T}.$$
From the three estimates above, we have
\begin{equation}\label{e19-thm3.2}
\begin{aligned}
\E_{x,i}\1_{\{\xi\wedge\zeta\leq m_0T\}}\big[U(X(T_2\wedge\tau_h))\wedge\Delta\big]
=&\E_{x,i}\1_{\{\zeta<m_0T, \zeta\leq \xi\}}\big[U(X(T_2\wedge\tau_h))\wedge\Delta\big]\\
&+\E_{x,i}\1_{\{\xi<m_0T\wedge\zeta\}}\big[U(X(T_2\wedge\tau_h))\wedge\Delta\big]\\
\leq&  \E_{x,i}\1_{\{\xi\wedge\zeta<m_0T\}}U(X(\xi\wedge\zeta)).
\end{aligned}
\end{equation}
Similarly,
\begin{equation}\label{e20-thm3.2}
\begin{aligned}
\E_{x,i} \1_{\{m_0T\leq \xi\wedge\zeta<T_2\}}\big[U(X(T_2\wedge\tau_h))\wedge\Delta\big]
\leq
&e^{\theta(\bar c+M_g)T}\E_{x,i}\1_{\{m_0T\leq \xi\wedge\zeta<T_2\}}U(X(\xi\wedge\zeta))\\
\leq& e^{\theta\lambda_2m_0T}\E_{x,i}\1_{\{m_0T\leq \xi\wedge\zeta<T_2\}} U(X(\xi\wedge\zeta)),
\end{aligned}
\end{equation}
where the last line follows from $m_0\lambda_2>\bar c+M_g+1$.
Applying \eqref{e19-thm3.2} and
\eqref{e20-thm3.2} to \eqref{e18-thm3.2}, we obtain
\begin{equation*}
\begin{aligned}
\E_{x,i} \big[U(X(T_2\wedge\tau_h))\wedge\Delta\big]\leq U(x)\, \text{ for any }(x,i)\in B_h\times\Z.
\end{aligned}
\end{equation*}
Since $\E_{x,i} \big[U(X(T_2\wedge\tau_h))\wedge\Delta\big]\leq \Delta$,
we have
\begin{equation}\label{e21-thm3.2}
\begin{aligned}
\E_{x,i} \big[U(X(T_2\wedge\tau_h))\wedge\Delta\big]\leq U(x)\wedge\Delta\, \text{  for any }(x,i)\in B_h\times\Z.
\end{aligned}
\end{equation}
This together with the Markov property of $(X(t),\alpha(t))$ implies
that $\{M(k):=\big[U(X(kT_2\wedge\tau_h))\wedge\Delta\big],k\in\N\}$
is a super-martingale.
Let $\eta=\inf\{k\in\N: M(k)=\Delta\}$.
Clearly, $\{\eta<\infty\}\supset\{\tau_h<\infty\}$.
For any $\eps>0$,
if $U(x)<\eps\Delta$ we have that
\begin{equation}\label{e36-thm3.2}
\PP_{x,i}\{\eta<k\}\leq \dfrac{\E_{x,i}M(\eta\wedge k)}\Delta\leq \dfrac{U(x)}{\Delta}\leq\eps.
\end{equation}
Letting $k\to\infty$ we have $$\PP_{x,i}\{\tau_h<\infty\}\leq\PP_{x,i}\{\eta<\infty\}\leq\eps,\,\text{ if }\, U(x)<\eps\Delta.$$
We complete the proof of this step by noting that
$\{x: U(x)<\eps\Delta\}$ is a neighborhood of $x$
due to the fact that $\lim_{x\to0}U(x)=0$.

{\bf Step 2: Asymptotic stability and pathwise convergence rate.}

To prove the asymptotic stability in probability,
we fix $h>0$ and define $U(x), T_2,m_0,\Delta$ depending on $h$ as in the first step.
By virtue of \eqref{e18-thm3.2}, we have
\begin{equation}\label{e22-thm3.2}
\begin{aligned}
\E_{x,i} e^{\theta\lambda_2(T_2\wedge\xi\wedge\zeta)}U(X(T_2\wedge\xi\wedge\zeta))
\geq&\E_{x,i} \1_{\{\xi\wedge\zeta<m_0T\}}U(X(\xi\wedge\zeta))\\
&+e^{\theta\lambda_2m_0T}\E_{x,i} \1_{\{m_0T\leq \xi\wedge\zeta<T_2\}}U(X(\xi\wedge\zeta))\\
&+e^{\theta\lambda_2T_2}\E_{x,i} \1_{\{\xi\wedge\zeta\geq T_2\}}U(X(T_2))\\
\geq&\E_{x,i} \1_{\{\xi<m_0T,\zeta>\xi\}}U(X(\xi))\\
&+e^{\theta\lambda_2m_0T}\E_{x,i} \1_{\{m_0T\leq \xi<T_2,\zeta>\xi\}}U(X(\xi))\\
&+e^{\theta\lambda_2T_2}\E_{x,i} \1_{\{\xi\wedge\zeta\geq T_2\}}U(X(T_2)).
\end{aligned}
\end{equation}
Recalling that $\zeta\leq\tau_h$ and $X(t)<h_2$ if $t<\zeta$, we have from \eqref{e15-thm3.2} and \eqref{e26-thm3.2} that
\begin{equation}\label{e23-thm3.2}
\begin{aligned}
\E_{x,i} \1_{\{\zeta\geq T_2\}}\1_{\{\xi<m_0T\}}U(X(T_2))
\leq &\E_{x,i} \1_{\{\xi<\zeta\}}\1_{\{\xi<m_0T\}}U(X(T_2\wedge\tau_h))\\
\leq&\E_{x,i}\left[\1_{\{\xi<m_0T\wedge\zeta\}}U(X(\xi))\exp\left\{-\theta\dfrac{\lambda}8(T_2-\xi)\right\}\right]\\
\leq&\exp\left\{-\dfrac{\theta\lambda T}8\right\}\E_{x,i}\left[\1_{\{\zeta\geq \xi\}}\1_{\{\xi<m_0T\}}U(X(\xi))\right]\\
\end{aligned}
\end{equation}
and
\begin{equation}\label{e24-thm3.2}
\begin{aligned}
\E_{x,i} \1_{\{m_0T\leq \xi<T_2,\zeta>T_2\}}U(X(T_2))
\leq &\E_{x,i} \1_{\{m_0T\leq \xi<T_2\wedge\zeta\}}U(X(T_2\wedge\tau_h))\\
\leq&\E_{x,i}\left[\1_{\{m_0T\leq \xi<T_2\wedge\zeta\}}U(X(\xi))\exp\left\{\theta(\bar c+M_g)(T_2-\xi)\right\}\right]\\
\leq&\exp\left\{\theta(\bar c+M_g)T\right\}\E_{x,i}\left[\1_{\{m_0T\leq \xi<T_2\wedge\zeta\}}U(X(\xi))\right]\\
\leq&\exp\{-\theta T\}\exp\left\{\theta\lambda_2m_0T\right\}\E_{x,i}\left[\1_{\{m_0T\leq \xi<T_2\wedge\zeta\}}U(X(\xi))\right].
\end{aligned}
\end{equation}
On the other hand, we can write
\begin{equation}\label{e25-thm3.2}
\begin{aligned}
\E_{x,i} \1_{\{\xi\wedge\zeta\geq T_2\}}U(X(T_2))= e^{-\theta\lambda_2T_2}e^{\theta\lambda_2T_2}\E_{x,i} \1_{\{\xi\wedge\zeta\geq T_2\}}
U(X(T_2)).
\end{aligned}
\end{equation}
Letting $p=\max\left\{\exp\left\{-\dfrac{\theta\lambda T}8\right\}, \exp\{-\theta T\}, \exp\{-\theta\lambda_2T_2\}\right\}<1$
and adding \eqref{e23-thm3.2}, \eqref{e24-thm3.2}, and \eqref{e24-thm3.2} side by side and then using \eqref{e22-thm3.2} we have
$$\E_{x,i} \1_{\{\zeta\geq T_2\}}U(X(T_2))\leq p U(x),\text{ for } (x,i)\in B_h\times\N.$$
By the strong Markov property of the process $(X(t),\alpha(t))$,
$$
\begin{aligned}
\E_{x,i} \1_{\{\zeta\geq 2T_2\}}U(X(2T_2))=&\E_{x,i}\left[\1_{\{\zeta\geq T_2\}}\E_{X(T_2),\alpha(T_2)} \1_{\{\zeta\geq T_2\}}U(X(T_2))\right]\\
\leq& p\E_{x,i} \1_{\{\zeta\geq T_2\}}U(X(T_2))\\
\leq& p^2 U(x),\text{ for } (x,i)\in B_h\times\N.
\end{aligned}
$$
Continuing this way we have
$$\E_{x,i} \1_{\{\zeta\geq kT_2\}}U(X(kT_2))\leq p^k U(x),\text{ for } (x,i)\in B_h\times\N.$$
Since $2\theta<1$, we have from \eqref{e12-thm3.2} that
$\E_{x,i} e^{2\theta H(s)}\leq e^{2\theta[\bar c+M_g]s}$.
This and  the Burkholder-Davis-Gundy inequality imply
\begin{equation}\label{e28-thm3.2}
\begin{aligned}
\E_{x,i}\sup_{t\leq T_2}&\left|\int_0^{t\wedge\tau_h}e^{\theta H(s)}\dfrac{V_x(X(s))\sigma(X(s),\alpha(s))}{g(V(X(s)))}dW(s)\right|\\
\leq& \left[\E_{x,i}\int_0^{T_2\wedge\tau_h}e^{2\theta H(s)}\dfrac{|V_x(X(s))\sigma(X(s),\alpha(s))|^2}{g^2(V(X(s)))}ds\right]^{\frac12}\\
\leq& \left[M_g^2\E_{x,i}\int_0^{T_2}e^{2\theta H(s)}ds\right]^{\frac12}\\
\leq& \left[M_g^2\int_0^{T_2}e^{2\theta[\bar c+M_g]s}ds\right]^{\frac12}:=\wdt K_1.
\end{aligned}
\end{equation}
On the other hand
\begin{equation}\label{e29-thm3.2}
\begin{aligned}
\E_{x,i}\sup_{t\leq T_2}&\left|\int_0^{t\wedge\tau_h}e^{\theta H(s)}\left[\theta c(\alpha(s))+\dfrac{\theta^2}2\dfrac{|V_x(X(s))\sigma(X(s),\alpha(s))|^2}{2g^2(V(X(s)))}\right]ds\right|\\
\leq& (\bar c+M_g)\E_{x,i}\int_0^{T_2\wedge\tau_h}e^{\theta H(s)}ds\\
\leq& (\bar c+M_g)\int_0^{T_2}e^{\theta [\bar c+M_g]}ds:=\wdt K_2.
\end{aligned}
\end{equation}
It follows from \eqref{e28-thm3.2} and \eqref{e29-thm3.2} that
\begin{equation}\label{e30-thm3.2}
\begin{aligned}
\E_{x,i}\sup_{t\leq T_2} U(X(t\wedge\tau_h))=&U(x)\E_{x,i}\sup_{t\leq T_2} e^{\theta H_t}\\
\leq& U(x)[1+\wdt K_1+\wdt K_2]:=U(x)\wdt K_3.
\end{aligned}
\end{equation}
By the strong Markov property of $(X(t),\alpha(t))$,
we derive from \eqref{e30-thm3.2} that
\begin{equation}\label{e31-thm3.2}
\begin{aligned}
\E_{x,i}&\1_{\{\zeta=\infty\}}\sup_{t\in[kT_2,(k+1)T_2]}U(X(t\wedge\tau_h))\\
&\leq\E_{x,i}\1_{\{\zeta\geq kT_2\}}\sup_{t\in[kT_2,(k+1)T_2]}U(X(t\wedge\tau_h))\\
&\leq\wdt  K_3\E_{x,i}\1_{\{\zeta\geq k T_2\}}U(X(kT_2)\\
&\leq\wdt  K_3U(x)\rho^k,
\end{aligned}
\end{equation}
which combined with Markov's inequality leads to
\begin{equation}\label{e32-thm3.2}
\begin{aligned}
\PP_{x,i}&\left\{\1_{\{\zeta=\infty\}}\sup_{t\in[kT_2,(k+1)T_2]}U(X(t\wedge\tau_h))> (\rho+\wdt\eps)^k\right\}\\
&\leq \dfrac1{(\rho+\wdt\eps)^k}\E_{x,i}\left[\1_{\{\zeta=\infty\}}\sup_{t\in[kT_2,(k+1)T_2]}U(X(t\wedge\tau_h))\right]\\
&\leq\wdt  K_3U(x)\dfrac{\rho^k}{(\rho+\wdt\eps)^k}\,\quad k\in\N,
\end{aligned}
\end{equation}
where $\wdt\eps$ is any number in $(0, 1-\rho)$.
In view of the Borel-Cantelli lemma,
for almost all $\omega\in\Omega$,
there exists random integer $k_1=k_1(\omega)$ such that
$$
\1_{\{\zeta=\infty\}}\sup_{t\in[kT_2,(k+1)T_2]}U(X(t))< (\rho+\wdt\eps)^k\,\text{ for any } k\geq k_1.
$$
Thus, for almost all $\omega\in\{\zeta=\infty\},$
we have
\begin{equation}\label{e33-thm3.2}
\begin{aligned}
G(V(X(t)))\leq [t/T_2]\ln (\rho+\wdt\eps)\leq -\lambda t \, \text{ for } t\geq k_1 T_2.
\end{aligned}
\end{equation}
where
$[t/T_2]$ is the integer part of $t/T_2$ and
$\lambda=-\dfrac{\ln (\rho+\wdt\eps)}{2T_2}>0$.
Since $G(y)$ is decreasing and maps $(0,h]$ onto $(-\infty, 0]$,
\eqref{e-rate} follows from \eqref{e36-thm3.2} and \eqref{e33-thm3.2}.
\end{proof}

In Theorem \ref{thm3.2},
under the condition that $\alpha(t)$ is merely ergodic,
we need and additional condition \eqref{e3-thm3.2}
to obtain the stability in probability of the system.
If $\alpha(t)$ is strongly ergodic,
the condition \eqref{e3-thm3.2}
can be removed.
\begin{thm}\label{thm3.3}
Suppose that
\begin{itemize}
\item for any $T>0$ and a bounded function $f:\N\mapsto\R$, we have
\begin{equation}\label{e1-thm3.3}
\lim_{x\to0} \sup_{i\in\N}\left\{\left|\E_{x,i}\int_0^Tf(\alpha(s))ds-\E_{i}\int_0^Tf(\hat\alpha(s)ds\right|\right\}=0.
\end{equation}
\item Assumption \ref{asp3.3} is satisfied
\item
the Markov chain $\hat\alpha(t)$
is strongly ergodic with invariant probability measure $\bnu=(\nu_1,\nu_2,\dots)$.
\end{itemize}
Suppose further that \eqref{e4-thm3.2} is satisfied and $\sum_{i\in\N} c_i\nu_i<0$.
Then the conclusion of Theorem \ref{thm3.2} holds.
\end{thm}

\begin{rem}\label{rem3.3}{\rm
We will prove in the Appendix that  \eqref{e1-thm3.3}
holds if
Assumption
\ref{asp3.3} and \eqref{e1-thm3.1} hold.
}
\end{rem}

\begin{proof}[Proof of Theorem \ref{thm3.3}]
Let $\lambda=-\sum_{i\in\N} c_i\nu_i$. Because of the uniform ergodicity of $\hat\alpha(t)$,
there exists a $T>0$  such that
\begin{equation}\label{e8-thm3.3}
\E_{0,i}\int_0^{t}c(\alpha(s))ds=\E_{i}\int_0^{t} c(\hat\alpha(s))ds\leq -\dfrac{3\lambda}4t\quad\,\forall\, t\geq T, i\in\N.
\end{equation}
By \eqref{e1-thm3.3},
there exists $h_1\in(0,h)$ such that
\begin{equation}\label{e9-thm3.3}
\E_{x,i}\int_0^{T}c(\alpha(s))ds\leq -\dfrac{\lambda}2T\quad\,\forall\,  |x|\leq h_1, i\in\N.
\end{equation}
In view of Lemma \ref{lm3.2},
there exists $h_2\in(0,h_1)$ such that
\begin{equation}\label{e10-thm3.3}
\bar c\PP_{x,i}\{\tau_h<T\}\leq \dfrac{\lambda}4\,\text{ provided } |x|\leq h_2, i\in\N.
\end{equation}
Applying \eqref{e9-thm3.3} and \eqref{e10-thm3.3} to \eqref{e7-thm3.2}, we have
\begin{equation}\label{e11-thm3.3}
\E_{x,i}H(T)\leq -\dfrac{\lambda}4T \,\text{ if }\, 0<|x|\leq h_2, i\in\N.
\end{equation}
Using \eqref{e11-thm3.3}, we can use arguments in the proof Theorem  \ref{thm3.2} to show that
\begin{equation}\label{e12-thm3.3}
\E_{x,i}e^{\theta H(T)}\leq \exp\left\{-\dfrac{\theta\lambda T}8\right\}\,\text{ for }0<|x|<h_2
\end{equation}
for a sufficiently small $\theta>0$.
This implies that
\begin{equation}\label{e13-thm3.3}
\E_{x,i}U\big(X(T\wedge\tau_h)\big)\leq  \exp\left\{-\dfrac{\theta\lambda T}8\right\} U(x),
\end{equation}
where $U(x)=\exp(\theta G(V(x))).$
Thus, $\{M_k:=U\big(X\big((kT)\wedge\tau_h\big)\big),k=0,1,\dots\}$
is a bounded supermartingale.
Then we can easily obtain the stability in probability of the trivial solution.
Moreover, proceeding as in Step 2 of the proof of Theorem \ref{thm3.2}, we can obtain the asymptotic stability as well as the rate of convergence.
The arguments are actually simpler because \eqref{e13-thm3.3} holds uniformly in $i\in\N$,
rather than $i\in\{1,\dots,k_0\}$ in the proof of Theorem \ref{thm3.2}.
\end{proof}
\begin{rem}{\rm
Consider the special case $g(y)\equiv y$.
With this function we have $U(X(t))=V(X(t))$.
Thus, if Assumption \ref{asp3.3} holds with $g(y)\equiv y$,
then the conclusion on stability in Theorems \ref{thm3.2} and \ref{thm3.3}
are still true without the condition \eqref{e4-thm3.2}
because we still have $\E V(X(t\wedge\tau_h))\leq V(x)e^{\bar ct}$,
which can be used in place of \eqref{e26-thm3.2}.
However, in order to obtain asymptotic stability and rate of convergence,
\eqref{e4-thm3.2} is needed.
In that case, if the initial value is sufficiently closed to $0$,
$V(X(t))$ will converges exponentially fast to $0$ with a large probability.

}
\end{rem}

\begin{thm}\label{thm3.4}
Consider the case that the state space of $\alpha(t)$ is finite, say $\M=\{1,\dots,m_0\}$
 for some positive integer $m_0$, rather than $\N$.
Suppose that $Q(0)$ is irreducible and let $\bnu$ be the invariant probability measure
of the Markov chain with generator $Q(0)$.
If  $\sum_{i\in\M} c_i\nu_i<0$ then
the trivial solution is asymptotically stable in probability, and for any $\eps>0$, there are $\lambda>0,\delta>0$ such that
$$\PP_{x,i}\left\{\lim_{t\to\infty} \dfrac{V(X(t))}{G^{-1}(-\lambda_3t)}\leq 1\right\}>1-\eps\,\text{ for any } (x,i)\in B_\delta\times\N.$$\end{thm}
We now provide some conditions for instability in probability.
\begin{thm}\label{thm3.5}
Suppose that the Markov chain $\hat\alpha(t)$
is ergodic with invariant probability measure $\bnu=(\nu_1,\nu_2,\dots)$
and that there are functions $g\in \Gamma,$
$V: D\mapsto\R_+$ such that
\begin{itemize}
\item $V(x)=0$ if and only if $x=0$
\item $V(x)$ is continuous on $D$ and twice continuously differentiable in $D\setminus\{0\}$.
\item there is a bounded  sequence of real numbers $\{c_i: i\in\N\}$
such that
\begin{equation}
\LL_i V(x)\geq c_ig(V(x)) \,\forall\, x\in D\setminus\{0\}.
\end{equation}
\end{itemize}
If
\eqref{e4-thm3.2}
is satisfied and if $\sum_{i\in\N} c_i\nu_i<0$
and
$\limsup_{i\to\infty}c_i<0,$
then
the trivial solution is unstable in probability.
\end{thm}
\begin{proof}
Define
$G(y)=-\int_y^hg^{-1}(z)dz$ as in Theorem \ref{thm3.2}.
We have from It\^o's formula,
\begin{equation}\label{e6-thm3.2}
\begin{aligned}
-G\big(V(X(\tau_h\wedge t))\big)=&-G(V(x))-
\int_0^{\tau_h\wedge t}\dfrac{\LL_{\alpha(s)} V(X(s))}{g(V(X(s)))}ds\\
&+\int_0^{\tau_h\wedge t}\dfrac{\dot g(V(X(s)))|V_x(X(s))\sigma(X(s),\alpha(s))|^2}{2g^2(V(X(s)))}ds\\
&-\int_0^{\tau_h\wedge t}\dfrac{V_x(X(s))\sigma(X(s),\alpha(s))}{g(V(X(s)))}dW(s)
\leq -G(V(x))+\wdt H(t)
\end{aligned}
\end{equation}
where
$$
\begin{aligned}
\wdt H(t)=&-\int_0^{\tau_h\wedge t}c(\alpha(s))ds
-\int_0^{\tau_h\wedge t}\dfrac{V_x(X(s))\sigma(X(s),\alpha(s))}{g(V(X(s)))}dW(s).
\end{aligned}
$$
Then using \eqref{e6-thm3.2} and
proceeding in the same manner as in the proof of Theorem \ref{thm3.2}
with $H(t)$ replaced with $\wdt H(t)$,
we have can find a sufficiently small $\wdt\theta, \wdt\Delta>0$ and a sufficiently large $T_3>0$ such that
$$\E_{x,i} \1_{\{\wdt\zeta\geq kT_3\}}\wdt U(X(kT_2))\leq p^k\wdt U(x),\text{ for } (x,i)\in B_h\times\N.$$
where
$\wdt U(x)=\exp\big\{-\wdt\theta G(V(x))\big\}$,
and
$\wdt\zeta=\inf\{k\geq 0: U(X(kT_3))\leq\wdt\Delta^{-1}\}$.
Note that, unlike $U(x)$, we have $\lim_{x\to0} \wdt U(x)=\infty$.
Since $U(X(kT_3))\geq\wdt\Delta^{-1}$
if $\wdt\zeta\geq k$,
we have that
$$\PP_{x,i}\{\wdt\zeta=\infty\}=\lim_{k\to\infty}\PP_{x,i}\{\wdt\zeta\geq k\}=0.$$
\end{proof}
Similarly, we can obtain a counterpart of Theorem \ref{thm3.3} for instability.
\begin{thm}\label{thm3.5}
Suppose that the Markov chain $\hat\alpha(t)$
is strongly ergodic with invariant probability measure $\bnu=(\nu_1,\nu_2,\dots)$
and that there are functions $g\in \Gamma,$
$V: D\mapsto\R_+$ such that
\begin{itemize}
\item $V(x)=0$ if and only if $x=0$
\item $V(x)$ is continuous on $D$ and twice continuously differentiable in $D\setminus\{0\}$.
\item there is a bounded  sequence of real numbers $\{c_i: i\in\N\}$
such that
\begin{equation}
\LL_i V(x)\geq c_ig(V(x)) \,\forall\, x\in D\setminus\{0\}.
\end{equation}
\end{itemize}
If
\eqref{e4-thm3.2}
and
\eqref{e1-thm3.1}
are satisfied and if $\sum_{i\in\N} c_i\nu_i>0$
then
the trivial solution is  unstable in probability.
\end{thm}

\section{Linearized Systems}\label{sec:lin}
Suppose  that \eqref{e1-thm3.3} is satisfied and that $\hat\alpha(t)$
is  a strongly ergodic Markov chain.
\begin{asp}\label{asp4.1}
Suppose that for $i\in\N$, there exist $b(i),\sigma_k(i)\in\R^{n\times n}$ bounded uniformly for $i\in\N$ such that
$$\xi_i(x):=b(x,i)-b(i)x,\quad \zeta_{i}(x):=\sigma(x,i)-(\sigma_1(i)x,\dots,\sigma_d(i)x)$$
satisfying
\begin{equation}\label{e3-ex2}
\lim_{x\to0} \sup_{i\in\N}\left\{\dfrac{|\xi_i(x)|\vee |\zeta_i(x)|}{|x|}\right\}=0.
\end{equation}
\end{asp}
For $i\in\N$, $k\in\{1,\dots,n\}$,
let $\Lambda_{1,i}$ and $\Lambda_{2,i,k}$ be the maximum eigenvalues of
$\dfrac{b(i)+b^\top(i)}{2}$ and $\sigma_k(i)\sigma_k^\top(i)$
respectively.
Similarly,
denote by $\lambda_{1,i}$ and $\lambda_{2,i,k}$ be the minimum eigenvalues of
$\dfrac{b(i)+b^\top(i)}{2}$ and $\sigma_k(i)\sigma_k^\top(i)$
respectively.

Suppose that $\Lambda_{1,i}$ and $\Lambda_{2,i,k}$ are bounded in $i\in\N$
then
we claim that
if $$\sum_{i\in\N}\nu_i\left(\Lambda_{1,i}+\dfrac12\sum_{k=1}^n\Lambda_{2,i,k}\right)<0,$$
then the trivial solution is asymptotic stable.

To show that, let $\eps>0$ be sufficiently small such that
\begin{equation}\label{e4-ex2}
\sum_{i\in\N}\nu_i\left(\eps+\Lambda_{1,i}+\dfrac12\sum_{k=1}^n\Lambda_{2,i,k}\right)<0.
\end{equation}

Define $V(x)=|x|^p$, carry out the calculation and obtain the estimates
as in that of \cite[Theorem 4.3]{KZY},
we can find a sufficiently small $p>0$ and $\hbar>0$ such that
\begin{equation}\label{e5-ex2}
\LL_iV(x)\leq p\left(\eps+\Lambda_{1,i}+\dfrac12\sum_{k=1}^n\Lambda_{2,i,k}\right)V(x)\,\text{ for } 0<|x|<\hbar.
\end{equation}
(Note that the existence of such $p$ and $\hbar$
satisfying \eqref{e5-ex2} uniformly for $i\in\N$ is due to \eqref{e5-ex2} and the boundedness of $\Lambda_{1,i}$ and $\Lambda_{2,i,k}$.)

By \eqref{e4-ex2} and \eqref{e5-ex2},
it follows from Theorem \ref{thm3.3} that
the trivial solution is asymptotic stable
and for any $\eps>0$, there exists $\delta>0$, $\lambda>0$ such that
$$
\PP_{x,i}\left\{ \lim_{t\to\infty} e^{\lambda t}|X(t)|\leq 1\right\}\geq 1-\eps\,\text{ for } (x,i)\in B_\delta\times\N.
$$

Similarly, if
$\sum_{i\in\N}\nu_i\left(\lambda_{1,i}+\dfrac12\sum_{k=1}^n\lambda_{2,i,k}\right)>0,$
and if $\lambda_{1,i}$ and $\lambda_{2,i,k}$ are bounded in $i\in\N, k=1,\dots,n$,
we have that the trivial solution is unstable.
To sum up,
we have the following result.

\begin{prop}
Let Assumption \ref{asp4.1} is satisfied.
We claim that,
\begin{itemize}
  \item if $\Lambda_{1,i}$ and $\Lambda_{2,i,k}$ are bounded in $i\in\N$ and $$\sum_{i\in\N}\nu_i\left(\Lambda_{1,i}+\dfrac12\sum_{k=1}^n\Lambda_{2,i,k}\right)<0,$$
  then the trivial solution is asymptotically stable in probability;
  \item if $\lambda_{1,i}$ and $\lambda_{2,i,k}$ are bounded in $i\in\N, k=1,\dots,n$,
and
$$\sum_{i\in\N}\nu_i\left(\lambda_{1,i}+\dfrac12\sum_{k=1}^n\lambda_{2,i,k}\right)>0,$$
then the trivial solution is unstable in probability.
\end{itemize}
\end{prop}

\section{Examples}\label{sec:exa}
This section provides several examples.

\begin{exam}\label{ex1}
Consider a real-valued switching diffusion
\begin{equation}\label{e1-ex1}
dX(t)=b(\alpha(t))X(t)[|X(t)|^\gamma\vee1]dt+\sigma(\alpha(t))\sin^2 X(t)dW(t),\,\, 0<\gamma<1,
\end{equation}
where $a\vee b = \max (a,b)$ for two real numbers $a$ and $b$,
and $Q(x)=\big(q_{ij}(x)\big)_{\N\times\N}$
with
$$q_{ij}(x)=
\begin{cases}
-\check p_1(x)&\,\text{ if } i=j=1\\
\check p_1(x)&\,\text{ if } i=1,j=2\\
-\hat p_i(x)-\check p_i(x)&\,\text{ if }i=j\geq2\\
\hat p_i(x)&\,\text{ if } i\geq 2, j=i-1\\
\check p_i(x)&\,\text{ if } i\geq 2, j=i+1.
\end{cases}
$$
Note that the drift grow faster than linear and the diffusion coefficient
is locally like $x^2$ near the origin for the continuous state.
Suppose that $b(i),\sigma(i),\check p_i(x),\hat p_i(x)$ are bounded for $(x,i)\in\R\times\N$
and $\check p_i(x),\hat p_i(x)$ are continuous in $\R^n$ for each
$i\in\N$.
It is well known (see \cite[Chapter 8]{WA}) that if
$$\nu^*:=\sum_{k=2}^\infty \prod_{\ell=2}^k\dfrac{\check p_{\ell-1}(0)}{\hat p_\ell(0)}<\infty,$$
then $\hat\alpha(t)$ is ergodic with the invariant measure $\bnu$ given by
$$\nu_1=\dfrac1{\nu^*},\,\nu_k=\dfrac1{\nu^*}\prod_{\ell=2}^k\dfrac{\check p_{\ell-1}(0)}{\hat p_\ell(0)}, k\geq 2.$$
We suppose that
$$\sum b(i)\nu_i<0,\,\,\text{ and }\, \limsup_{i\to\infty} b(i)<0.$$
we will show that the trivial solution is stable.
Let $0<\eps<-\sum b(i)\nu_i$
then
$\sum [b(i)+\eps]\nu_i<0$.
Let $$V(x)=x^2$$
We have
$$\LL_i V(x)=2b(i)|x|^{2+2\gamma}+\sigma^2(i)\sin^4(x)$$
Since $\gamma<1$ and $\sigma(i)$ is bounded, there exists an $\hbar>0$ such that
$\sigma^2(i)\sin^4(x)\leq \eps|x|^{2+2\gamma}$
given that $|x|\leq \hbar$.
then
$$\LL_i V(x)\leq [2b(i)+\eps]|x|^{2+2\gamma}=[2b(i)+\eps] V^{1+\gamma}(x)\,\text{ in } [-\hbar,\hbar]\times\N.$$
By Theorem \ref{thm3.2}, the trivial solution is asymptotically stable in probability.
Moreover,
for the function $g(y)=y^{1+\gamma}$,
$$G(y):=-\int_y^\hbar \dfrac1{g(z)}ds=\hbar^{-\gamma}-y^{-\gamma},\, y\in(0,\hbar]$$
has the inverse
$$G^{-1}(-t)=\dfrac{1}{\left[t+\hbar^{-\gamma}\right]^{1/\gamma}},\, \text{ for } t\geq0$$
Thus, for any $\eps>0$,
there exists a $\delta>0$ such that
if $(x,i)\in [0,\delta]\times\N$,
then, there exists a $\lambda>0$ such that
$$\PP_{x,i}\left\{\limsup_{t\to\infty} t^{1/\gamma}X^2(t)\leq \lambda\right\}>1-\eps.$$
\end{exam}

\begin{exam}\label{ex2}
This example consider a
random-switching linear systems of differential equations:
\begin{equation}\label{e1-ex2}
d X(t)=A(\alpha(t))X(t)dt
\end{equation}
where $A(i)\in\R^{n\times n}$
satisfying $\sup_{i\in\N}\{|\lambda_i|\vee|\Lambda_i|\}<\infty$
with $\lambda_i,\Lambda_i$ being the minimum and maximum eigenvalues
of $A(i)$, respectively.
Let
$$
Q(x)=\left(\begin{array}{cccccc}
-1-\sin|x|&1+\sin|x|&0&0&0&\cdots\\
1+\sin|x|&-2-2\sin|x|&1+\sin|x|&0&0&\cdots\\
1+\sin|x|&0&-2-2\sin|x|&1+\sin|x|&0&\cdots\\
1+\sin|x|&0&0&-2-2\sin|x|&1+\sin|x|&\cdots\\
\vdots&\vdots&\vdots&\vdots&\vdots&\ddots\\
\end{array}\right).
$$
By \cite[Proposition 3.3]{WA},
it is easy to verify that the Markov chain
$\hat\alpha(t)$ with generators $Q(0)$
is strongly ergodic.
Solving the system
$$\bnu Q(0)=0, \sum\bnu_i=1$$
we obtain that the invariant measure of $\hat\alpha(t)$
is
$(\nu_i)_{i=1}^\infty=(2^{-i})_{i=1}^\infty.$
Thus,
if
$\sum \lambda_i2^{-i}>0$ the trivial solution to \eqref{e1-ex2} is unstable.
In case $\sum \Lambda_i 2^{-i}<0$ the trivial solution to \eqref{e1-ex2} is asymptotically stable in probability.
In particular,
suppose that
$n=2$ and $A(i)$ are upper triangle matrices, that is,
$$A(i)=\left(\begin{array}{cc}
a_i&b_i\\
0&c_i\\
\end{array}\right).
$$
If $a_i$ and $c_i$ are positive for $i\geq 2$, then the system
$dX(t)=A(i)X(t)dt$
is unstable.
However, if $a_1, c_1<-\sup_{i\geq2}\{a_i, c_i\}$,
then $\sum (a_i\vee c_i)2^{-i}<0$.
Thus, the switching differential system is asymptotically stable.
The stability of the system at state 1 and the switching process
become a stabilizing factor.

On the other hand,
if $a_i\wedge c_i$ is negative for $i\geq 2$, then the system
$dX(t)=A(i)X(t)dt$
is asymptotically stable.
Suppose further that $a_1, c_1>\sup_{i\geq2}\{-(a_i\wedge c_i)\}$,
then $\sum (a_i\vee c_i)2^{-i}>0$.
Under this condition, the switching differential system is unstable.
\end{exam}

\section{Further Remarks}\label{sec:rem}

Using a new method, we provide sufficient
conditions for stability
 and instability in probability of a class of regime-switching diffusion systems with switching states belonging to a countable set.
The conditions are based on the relation of a ``switching-independent" Lyapunov function
and the generator of the switching part.

Although the systems under consideration are memoryless,
the main results of this paper hold if we assume that
the switching intensities $q_{ij}$ depend on
the history of $\{X(t)\}$ rather than the current state of $X(t)$, (see \cite{DY1, DY2}
for  fundamental properties of this process).
The problem can be formulated as follows.
Let $r$ be a fixed positive number.
Denote by $\C$ the set of  $\R^n$-valued continuous functions defined on $[-r, 0]$.
For $\phi\in\C$, we use the norm $\|\phi\|=\sup\{|\phi(t)|: t\in[-r,0]\}$.
For $t\geq0$, we denote  by $y_t$
the
so-called segment function (or memory segment function) $y_t =\{ y(t+s): -r \le s \le 0\}$.
We assume that the jump intensity of $\alpha(t)$ depends on the trajectory of $X(t)$ in the interval $[t-r,t]$.
That is, there are functions $q_{ij}(\cdot):\C\to\R$ for $i,j\in\N$
satisfying that
 $q_{i}(\phi):=\sum_{j=1,j\ne i}^\infty q_{ij}(\phi)$ is uniformly bounded in $(\phi,i)\in\C\times\N$ and that
$q_{i}(\cdot)$ and $q_{ij}(\cdot)$ are continuous such that
\begin{equation}\label{eq:tran-p}\begin{array}{ll}
&\disp \PP\{\alpha(t+\Delta)=j|\alpha(t)=i, X_s,\alpha(s), s\leq t\}=q_{ij}(X_t)\Delta+o(\Delta) \text{ if } i\ne j \
\hbox{ and }\\
&\disp \PP\{\alpha(t+\Delta)=i|\alpha(t)=i, X_s,\alpha(s), s\leq t\}=1-q_{i}(X_t)\Delta+o(\Delta).\end{array}\end{equation}

It is proved in \cite{DY1} that
if either Assumption \ref{asp4.1} or Assumption \ref{asp4.2}
is satisfied with $x, y\in\R^n$ replaced by $\phi,\psi\in\C$,
then there is a unique solution
to the switching diffusion \eqref{eq:sde} and \eqref{eq:tran-p}
with a given initial value.
Moreover, the process $(X_t,\alpha(t))$ has the Markov-Feller property.
With slight modification in the proofs,
the theorems in Section 3 still hold
for  system \eqref{eq:sde} and \eqref{eq:tran-p}.

Our method can also be applied to regime-switching jump diffusion processes.
The results obtain by using our method will
generalize existing results (e.g., \cite{XY,YX})
to the case of regime-switching jump diffusions with countable regimes.

On the other hand,
there is a gap between sufficient conditions for stability and instability in Proposition 4.1.
To overcome the difficulty, 
we need to make a polar coordinate transformation to
decompose 
of $X(t)$
into the radial part $r(t)=|X(t)|$ and the angular part $Y(t)=X(t)/r(t)$.
Then, the Lyapunov exponents with respect to invariant measures
of the linearized process of $(Y(t), \alpha(t))$
will determine whether or not the system is stable.
This approach has been used
to treat many linear and linearized stochastic systems
(e.g., \cite{AKO, PB, RK}).
In our setting,
the switching $\alpha(t)$
take values in a noncompact space,
thus, it is more difficult to examine invariant measures.
We will address this problem
together necessary conditions of stability 
 in a subsequent paper.

\appendix
\section{Appendix}\label{sec:apd}
\begin{proof}[Proof of Lemma \ref{lm3.2}]
Since $g$ is continuously differentiable and $g(0)=0$,
there is $K_g>0$ such that $g(z)\leq K_g |z|$ for $|z|\leq 1$.
Thus, we have
$$\LL_i V(x)\leq K_g\sup_{i\in\N}\{|c_i|\} V(x),\,\, (x,i)\in D\times\N$$
Letting $\tilde K=K_g\sup_{i\in\N}\{|c_i|\}$, by It\^o's formula,
$$\begin{aligned}
\E_{x,i} V(X(t\wedge \tau_h))\leq& E V(x)+\tilde K\E_{x,i}\int_0^{t\wedge\tau_h} V(X(s))ds\\
\leq& E V(x)+\tilde K\E_{x,i}\int_0^t\E_{x,i} V(X(s\wedge\tau_h))ds.
\end{aligned}
$$
By the Grownwall inequality, we can easily obtain
$$
\E_{x,i} V(X(T\wedge \tau_h))\leq V(x)e^{KT}.
$$
Since $V(0)=0$, standard arguments lead to the desired result.
\end{proof}

\begin{proof}[Proof of Lemma \ref{lm-a3}]
It is easy to show that there exists some $K_2>0$ such that
$$ |y|^k\exp(\theta y)\leq K_2(\exp(\theta_0y)+\exp(-\theta_0y)), k=1,2.$$
for $\theta\in \left[0,\frac{\theta_0}2\right]$, $y\in\R$.
For any $y\in\R$, let $\xi(y)$ be a number lying between $y$ and $0$ such that
$\exp(\xi(y))=\dfrac{e^y-1}y$.
Pick $\theta\in\left[0,\frac{\theta_0}2\right]$
and let $h\in\R$ such that $0\leq \theta+h\leq  \frac{\theta_0}2$. Then
$$\lim\limits_{h\to0}\dfrac{\exp((\theta+h) Y)-\exp(\theta Y)}h= Y\exp(\theta Y)\text{ a.s.,}$$
where $Y$ is as defined in Lemma \ref{lm-a3},
and
$$\left|\dfrac{\exp((\theta+h) Y)-\exp(\theta Y)}h\right|=|Y|\exp(\theta Y+\xi(hY))
\leq 2K_3[ \exp(\theta_0 Y)+\exp(-\theta_0 Y)].$$
By the Lebesgue dominated convergence theorem,
$$\dfrac{d \E \exp(\theta Y)}{d\theta}=\lim\limits_{h\to0}\E\dfrac{\exp((\theta+h) Y)-\exp(\theta Y)}h= \E Y\exp(\theta Y).$$
Similarly,
$$\dfrac{d^2 \E \exp(\theta Y)}{d\theta^2}=\E Y^2\exp(\theta Y).$$
As a result, we obtain
$$\dfrac{d\phi}{d\theta}=\dfrac{\E Y\exp(\theta Y)}{\E \exp(\theta Y)}$$
which implies
	$$\dfrac{d\phi}{d\theta}(0)=\E Y$$
and
$$\dfrac{d^2\phi}{d\theta^2}=\dfrac{\E Y^2\exp(\theta Y)\E \exp(\theta Y)-[\E Y\exp(\theta Y)]^2}{[\E \exp(\theta Y)]^2}.$$
By H{\"o}lder's inequality we have $\E Y^2\exp(\theta Y)\E \exp(\theta Y)\geq[\E Y\exp(\theta Y)]^2$
and therefore $$\dfrac{d^2\phi}{d\theta^2}\geq0\,,\forall\,\theta\in \left[0,\frac{\theta_0}2\right].$$
Moreover,
$$
\begin{aligned}
\dfrac{d^2\phi}{d\theta^2}\leq& \dfrac{\E Y^2\exp(\theta Y)}{\E \exp(\theta Y)}\\
\leq & \dfrac{K_3(\E \exp(\theta_0 Y)+\E \exp(-\theta_0 Y))}{\exp(\theta \E Y)}\\
\leq & \dfrac{K_3(\E \exp(\theta_0 Y)+\E \exp(-\theta_0 Y))}{\exp(-\theta_0|\E Y|)}:=K_2,
\end{aligned}
$$
which concludes the proof.
\end{proof}
\begin{proof}[Proof of Lemma \ref{lm3.3}]
Let $\bar\tau_n$
be the $n-$th jump moment
of $\alpha(t)$.
Let $T>0$,
In view of \cite[Lemma 4.3.2]{XM},
we have
$$\PP_{x,i}\{X(t)=0\,\text{ for some }\, t\in[0,T\wedge\bar\tau_1]\}=0
\,\text{ for any }\, x\ne 0, i\in\N.$$
Since
$X(T\wedge\bar\tau_1)\ne 0$ a.s.,
applying \cite[Lemma 4.3.2]{XM} again yields
$$\PP_{x,i}\{X(t)=0\,\text{ for some }\, t\in[T\wedge\bar\tau_1,T\wedge\bar\tau_2]\}=0
\,\text{ for any }\, x\ne 0, i\in\N.$$
Continuing this way, we have
\begin{equation}\label{e1-lm3.3}
\PP_{x,i}\{X(t)=0\,\text{ for some }\, t\in[0,T\wedge\bar\tau_n]\}=0
\,\text{ for any }\, x\ne 0, i\in\N, n\in\N.
\end{equation}
In \cite[Theorems 3.1 \& 3.3]{DY1},
we have that $\lim_{n\to\infty}\bar\tau_n=\infty$.
This and \eqref{e1-lm3.3}
imply
$$\PP_{x,i}\{X(t)=0\,\text{ for some }\, t\in[0,T]\}=0
\,\text{ for any }\, x\ne 0, i\in\N, n\in\N.
$$
Since $T$ is taken arbitrarily,
we obtain the desired result.
\end{proof}

\begin{lm}\label{lm-a1}
If the Markov chain $\hat\alpha(t)$ is strongly exponentially ergodic with generator $\hat Q$
and invariant probability measure $\bnu=(\nu_1,\nu_2,\dots)^\top$,
then if $\mathbf{b}=(b_1,b_2,\dots)^\top$ is bounded satisfying
$\sum \nu_ib_i=0$, then, there exists a bounded vector $\mathbf{c}=(c_1,c_2,\dots)^\top$ such that
$b_i=\sum \hat q_{ji}c_j$.
\end{lm}
\begin{proof}
Let $\hat P(t)=\hat p_{ij}(t)$, where $\hat p_{ij}(t)=\PP\{\hat\alpha(t)=j|\alpha(0)=i\}$, the transition matrix of $\hat\alpha(t)$.
Let $\mathbf{c}=(c_1,c_2,\dots)^\top$ where $c_i=\int_0^\infty [\nu_jb_j-\hat P_{ij}(t)b_i] dt$.
In view of \eqref{see}, it is easy to see that $\mathbf{c}$ is bounded.
Let $\mathbbm{1}=(1, 1, \dots)$. We have
$$
\begin{aligned}
\hat Q\mathbf{c}=&\int_0^\infty \big[\hat Q\bnu\mathbbm{1}\mathbf{b}-\hat Q\hat P(t)\mathbf{b}\big]dt\\
=&-\int_0^\infty \hat Q\hat P(t)\mathbf{b}dt\\
=&-\int_0^\infty\hat P(t)\mathbf{b}dt=-\hat P(t)\mathbf{b}\Big|_0^\infty\\
=&-\mathbbm{1}\bnu\mathbf{b}+\mathbf{b}=\mathbf{b}.
\end{aligned}
$$
\end{proof}

\begin{lm}\label{lm-a2}
Suppose that Assumption
\ref{asp3.3}
and \eqref{e1-thm3.1} hold.
Then
for any $T>0$ and a bounded function $f:\N\mapsto\R$, we have
\begin{equation}\label{e3-a2}
\lim_{x\to0} \sup_{i\in\N,t\in[0,T]}\left\{\left|\E_{x,i}f(\alpha(t))-\E_{i}f(\hat\alpha(t)\right|\right\}=0.
\end{equation}
\end{lm}

\begin{proof}
By the basic coupling method (see e.g., \cite[p. 11]{Chen}),
we can consider the joint process $(X(t),\alpha(t),\hat\alpha(t))$
as a switching diffusion
where the diffusion $X(t)\in\R^n$ satisfies
satisfying
\begin{equation}\label{e1-a3}
dX(t)=b(X(t), \alpha(t))dt+\sigma(X(t),\alpha(t))dw(t)
\end{equation}
and the
switching part $(\alpha(t),\hat\alpha(t))\in\N\times\N$ has the generator $\wdt Q(X(t))$ which is defined by
\begin{equation}
\begin{split}
\wdt  Q(x)\wdt  f(k,l)=&\sum_{j,i\in\N}\wdt q_{(k,l)(j,i)}(x)\left(\wdt  f(j,i)-\wdt  f(k,l)\right)\\
=&\sum_{j\in\N}[q_{kj}(x)-q_{lj}(0)]^+(\wdt  f(j,l)-\wdt  f(k,l))\\
&+\sum_{j\in\N}[q_{lj}(0)-q_{kj}(  x)]^+(\wdt  f(k,j)-\wdt  f(k,l))\\
&+\sum_{j\in\N}[q_{kj}(x)\wedge q_{lj}(0)](\wdt  f(j,j)-\wdt  f(k,l)).
\end{split}
\end{equation}
In what follows, we use
 the notation
 $\E_{x,i,j}$ and $\PP_{x,i,j}$ to denote the corresponding conditional expectation and probability for the coupled process
$(X(t),\alpha(t), \hat\alpha(t))$
conditioned on $(X(0),\alpha(0), \hat\alpha(0))=(x,i,j)$.
Let $\vartheta=\inf\{t\geq0: \alpha(t)\ne\hat\alpha(t)\}$.
Define $\tilde g:\Z\times\Z\mapsto\R$ by $\wdt g(k,l)=\1_{\{k=l\}}$.
By the definition of the function $\wdt g$,  we have
\begin{equation}\label{e3-a3}
\begin{aligned}
\wdt  Q(x)\wdt  g(k,k)=&\sum_{j\in\N, j\ne k}[q_{kj}(x)-q_{kj}(0)]^++\sum_{j\in\N, j\ne k}[q_{kj}(0)-q_{kj}(x)]^+\\
=&\sum_{j\in\N, j\ne k}|q_{kj}(x)-q_{kj}(0)|=:\Xi(x, k).
\end{aligned}
\end{equation}
For any $\eps>0$,
let $h>0$ such that $B_h\in D$ and
$\sup_{(x,k)\in B_h\times\N} \Xi(x,k)<\frac{\eps}{2T}$.
Applying  It\^o's formula and noting that $\alpha(t)=\hat \alpha(t), t\leq\vartheta$,

we obtain that
\begin{equation}\label{e4-a3}
\begin{split}
\PP_{x,i,i}\{\vartheta\leq T\wedge\tau_h\}=& \E_{x,i,i} \wdt g\left(\alpha(\vartheta\wedge T\wedge\tau_h),\hat \alpha(\vartheta\wedge T\wedge\tau_h)\right)\\
=&\E_{x,i,i}\int_0^{\vartheta\wedge T\wedge\tau_h}\wdt Q(X(t)) \wdt  g(\alpha(t),\hat \alpha(t))dt\\
=&\E_{x,i,i} \int_0^{\vartheta\wedge T\wedge\tau_h}\Xi(X(t), \alpha(t))dt\\
\leq& T\sup_{(x,i)\in B_h\times\N}\Xi(x,k)\leq \dfrac\eps2.
\end{split}
\end{equation}
Thus
In view of Lemma \ref{lm3.2},
there is $\delta>0$ such that
$\PP_{x,i,i}\{\tau_h\leq T\}\leq \dfrac\eps2$.
This and \eqref{e4-a3} derive
$$
\PP_{x,i,i}\{\vartheta\wedge\tau_h\leq T\}\leq \PP_{x,i,i}\{\vartheta\leq T\wedge\tau_h\}+\PP_{x,i,i}\{\tau_h\leq T\}\leq \eps.$$

We have that
$$
\begin{aligned}
\left|\E_{x,i}f(\alpha(t))-\E_{0,i}f(\alpha(t))\right|
=&\left|\E_{x,i,i}\left[f(\alpha(t))-f(\hat\alpha(t))\right]\right|\\
=&\left|\E_{x,i,i}\1_{\{\vartheta\wedge\tau_h\leq t\}}\left[f(\alpha(t))-f(\hat\alpha(t))\right]\right|\\
\leq &2M_f\PP_{x,i,i}\{\vartheta\wedge\tau_h\leq t\}\leq 2M_f\eps,\text{ for } t\in[0,T].
\end{aligned}
$$
where $M_f=\sup_{i\in\N} |f(i)|$.
The lemma is proved.
\end{proof}


\begin{thebibliography}{99}

\setlength{\baselineskip}{0.10in}

\parskip=0pt

\bibitem{WA} W.J. Anderson, {\it Continuous-time Markov chains: An Applications-Oriented Approach}, Springer, 2012.

\bibitem{AKO} L. Arnold, W. Kliemann, and E. Oeljeklaus. Lyapunov exponents of linear stochastic systems, {\it  Lyapunov exponents}. Springer Berlin Heidelberg, 1986. 85-125.

\bibitem{PB} P. Baxendale, Invariant measures for nonlinear stochastic differential equations, {\it Lyapunov
exponents (Oberwolfach, 1990)}, Lecture Notes in Math., vol. 1486, Springer, Berlin, 1991,
pp. 123–140.

\bibitem{BL}
M. Bena{\"\i}m and C. Lobry, Lotka {V}olterra in fluctuating environment
  or ``how switching between beneficial environments can make survival
  harder'', {\it Ann. Appl. Probab.} (2016), to appear.

\bibitem{Chen}
M.F. Chen, {\it From Markov Chains to Non-equilibrium Particle Systems}, World Scientific, Singapore, 2004.

\bibitem{RK} R. Z. Khasminskii, Necessary and sufficient conditions for the asymptotic stability of linear stochastic systems, {\it Theory Probab. Appl.} 12.1 (1967): 144-147.

\bibitem{KZY}
R.Z. Khasminskii, C. Zhu, and G. Yin, Stability of regime-switching diffusions, {\it Stochastic Process. Appl.}, {\bf117} (2007), no. 8, 1037-1051.

\bibitem{LS} R. Liptser and A.N. Shiryayev, {\it Theory of Martingales}, (Vol. 49) Springer, 2012.



\bibitem{XM} X. Mao, Stability of stochastic differential equations with Markovian switching, {\it Stochastic
Process. Appl.}, {\bf 79} (1999),45-67.


\bibitem{MY} X. Mao, C. Yuan. {\it Stochastic Differential Equations with Markovian Switching},  Imperial College Press, London, 2006.

\bibitem{DY1} D.H. Nguyen and G. Yin, Modeling and analysis of switching diffusion systems: past-dependent switching with a countable state space, {\it SIAM J. Control Optim. }{\bf54} (2016), no. 5, 2450-2477.

\bibitem{DY2} D.H. Nguyen and G. Yin, Recurrence and Ergodicity of Switching Diffusions with Past-Dependent Switching Having A Countable State Space, {\it submitted}





\bibitem{SX} J. Shao and F. Xi, Stability and recurrence of regime-switching diffusion processes, {\it SIAM J. Control Optim.} {\bf52} (2014), no. 6, 3496-3516.


\bibitem{AS} A.V. Skorokhod, {\it Asymptotic Methods in the Theory of Stochastic Differential Equations}, Vol. 78. American Mathematical Soc., 1989.

\bibitem{XY} F. Xi and G. Yin, Almost sure stability and instability for switching-jump-diffusion systems with state-dependent switching, {\it J. Math. Anal. Appl.}  {\bf400} (2013), no. 2, 460-474.

\bibitem{XZ} F. Xi and L. Zhao, On the stability of diffusion processes with state-dependent switching, {\it Sci.
China, Ser. A}, {\bf 49} (2006), 1258-1274


\bibitem{XZ} F. Xi and C. Zhu, On Feller and Strong Feller Properties and Exponential Ergodicity
of Regime-Switching Jump Diffusion Processes with Countable
Regimes, {\it submitted}.

\bibitem{YX} G. Yin and F.B. Xi,  Stability of regime-switching jump diffusions, {\it SIAM J. Control Optim.},
{\bf48} (2010), 4525-4549.


\bibitem{YZZ} G. Yin, B. Zhang, and C. Zhu, Practical stability and instability of regime-switching diffusions, {\it J. Control Theory Appl.} {\bf6} (2008), no. 2, 105-114.

\bibitem{YZW} G. Yin, G. Zhao, and F. Wu,  Regularization and stabilization of randomly switching dynamic systems {\it SIAM J. Appl. Math.} {\bf 72} (2012), no. 5, 1361-1382.


\bibitem{YZ} G. Yin  and C. Zhu, {\it Hybrid Switching Diffusions: Properties and Applications},  Springer, New York,  2010.

\bibitem{ZWYJ} X. Zong, F. Wu, G. Yin, and Z. Jin, Almost Sure and p$^{\rm th}$-Moment Stability and Stabilization of Regime-Switching Jump Diffusion Systems, {\it SIAM J. Control Optim.} {\bf52} (2014), no. 4, 2595-2622.


\end{thebibliography}
\end{document}